\documentclass[11pt]{article}

\usepackage{amsmath,amsfonts,amsthm,amscd,amssymb,graphicx,mathtools}
\mathtoolsset{showonlyrefs}
\usepackage{cite}
\usepackage{bbm}
\usepackage{color}
\usepackage[english=american]{csquotes}
\usepackage[colorlinks=true, pdfstartview=FitV, linkcolor=black, citecolor=black, urlcolor=black]{hyperref}
\usepackage{calc}
\usepackage{mathptmx}
\usepackage{t1enc}
\usepackage{stackrel}
\usepackage{MnSymbol}

\numberwithin{equation}{section}

\newtheorem{theorem}{Theorem}[section]

\newtheorem{lemma}[theorem]{Lemma}

\newtheorem{proposition}[theorem]{Proposition}
\newtheorem{prop}[theorem]{Proposition}

\newtheorem{remark}[theorem]{Remark}
\newtheorem{definition}[theorem]{Definition}

\DeclareMathOperator{\interior}{int}
\DeclareMathOperator{\diverg}{div}

\DeclareMathOperator{\supp}{supp}
\DeclareMathOperator{\kin}{kin}

\newcommand{\T}{\mathbb{T}}

\newcommand{\norm}[1]{\|#1\|}

\newcommand{\N}{\mathbb{N}}
\newcommand{\R}{\mathbb{R}}

\newcommand{\Z}{\mathbb{Z}}

\newcommand{\eps}{\epsilon}
\def\d{{\,\rm d}}

\renewcommand{\eps}{\varepsilon}
\renewcommand{\epsilon}{\varepsilon}
\renewcommand{\phi}{\varphi}

\begin{document}
	
	
	\title{A General Convex Integration Scheme for the Isentropic Compressible Euler Equations}
	
	\author{Tomasz D\k{e}biec\footnotemark[1] \and Jack Skipper\footnotemark[2] \and Emil Wiedemann\footnotemark[2]}
\date{}	
	
	\maketitle
	
	\begin{abstract}
	We prove via convex integration a result that allows to pass from a so-called subsolution of the isentropic Euler equations (in space dimension at least $2$) to exact weak solutions. The method is closely related to the incompressible scheme established by De Lellis--Sz\'ekelyhidi, in particular we only perturb momenta and not densities. Surprisingly, though, this turns out not to be a restriction, as can be seen from our simple characterization of the $\Lambda$-convex hull of the constitutive set. An important application of our scheme will be exhibited in forthcoming work by Gallenm\"uller--Wiedemann.  
		
		
	\end{abstract}
	
	\renewcommand{\thefootnote}{\fnsymbol{footnote}}
	\footnotetext[1]{Sorbonne Universit{\'e}, Inria, CNRS, Universit\'{e} de Paris, Laboratoire Jacques-Louis Lions (LJLL), F-75005 Paris, France. Email: tomasz.debiec@sorbonne-universite.fr}
	
	\footnotetext[2]{Institut f\"ur Angewandte Analysis, Universit\"at Ulm, Helmholtzstra\ss e 18, 89081 Ulm, Germany. Email:
		emil.wiedemann@uni-ulm.de, jwd.skipper@gmail.com}

	\section{Introduction}
	
	We consider the isentropic compressible Euler equations in $\R^d$, $d\geq 2$,
	\begin{equation}\label{eq:compressEulervelocity}
	\begin{aligned}
	\partial_t (\rho u)+\diverg_x\left(\rho u\otimes u\right)+\nabla_x p(\rho)&=0,\\
	\partial_t\rho+\diverg_x (\rho u)&=0,
	\end{aligned}
	\end{equation}
	and with the standard change of variables $m=\rho u $ we have the momentum formulation written
	\begin{equation}\label{eq:compressEuler}
	\begin{aligned}
	\partial_tm+\diverg_x\left(\frac{m\otimes m}{\rho}\right)+\nabla_x p(\rho)&=0,\\
	\partial_t\rho+\diverg_x m&=0,
	\end{aligned}
	\end{equation}
	where $p:[0,\infty)\to\R$ is a non-negative and strictly convex pressure function and we assume $\rho\geq\underline{\rho}>0$ for some constant $\underline{\rho}$. Note that one needs extra conditions to validate the second formulation when $\rho=0$.

In recent years, a lot of activity has arisen related to the ill-posedness of this system in the class of weak solutions. Even when an appropriate entropy condition is taken into account, there may be infinitely many weak solutions emanating from the very same initial data~\cite{admissible, chiodaroli, AkWi}. This is true for certain Riemann data~\cite{CDLK, markkling} and even for smooth data~\cite{CKMS}. In fact, the set of initial data that admits such non-unique solutions is dense in the energy space~\cite{CVY}.
Similar constructions also work for related systems, such as the full compressible Euler system in multiple space dimensions~\cite{CFK, AKKMM, KKMM}.

All these constructions rely on a technique known as {\em convex integration}, introduced to the context of fluid dynamics in the seminal work of De Lellis--Sz\'ekelyhidi~\cite{DLS2009, admissible}. There, an iteration scheme was presented for the incompressible Euler equations that allows to pass from a relaxed solution, also called subsolution, to an infinitude of exact solutions with prescribed energy profile. Since, in many situations, it is much easier to construct a subsolution than to construct an exact solution, this led to several breakthrough results in the existence theory of weak solutions~\cite{Wi11, vortexpaper, euleryoung} and eventually, through a series of refinements of the scheme, to a complete proof of Onsager's Conjecture~\cite{Isett, BDSV}. 

The incompressible convex integration scheme starts off from a subsolution and successively adds to it perturbations in the velocity that oscillate at a higher and higher frequency in each step, thus eventually producing a very irregular limit object that is then shown to be a weak solution of the Euler equations. So far, all known convex integration constructions for \emph{compressible} models recur to the \emph{incompressible} scheme, so that oscillations are added only to the velocity (or momentum) while the other dependent variable, the density, remains fixed along the iteration. This leads to solutions with highly irregular momentum but regular (e.g., $C^1$ or piecewise constant) densities. It has been believed by some, including the current authors, that there should be a more powerful convex integration scheme for the compressible Euler equations that would include oscillations both in the density and the momentum, und thus be genuinely compressible. The hope was that such a scheme would open up further possibilities for the existence theory of weak solutions. 

In this paper, we argue that this hope was not justified. Indeed, as a very simple yet remarkable result we give a characterization of the $\Lambda$-convex hull of the constitutive set for the isentropic Euler system (Theorem~\ref{prop:LambdaConvexHull}). The proof reveals that the representation of any state in the $\Lambda$-convex hull as a $\Lambda$-convex combination of states in the constitutive set is achievable \emph{by momentum oscillations alone}. In other words, oscillations in the density, or any other kind of oscillations, \emph{would not yield more solutions} than those obtained via momentum oscillations alone. Or, to rephrase it once more, the incompressible scheme is fully sufficient as far as convex integration for the isentropic system goes. 

On the positive side, we show two results (Theorems~\ref{prop:subsolutioncondition} and~\ref{prop:subsolutionconditiondens}) on the passage from subsolutions to solutions, closely following the method in~\cite{admissible}. It is shown, indeed, that any subsolution taking values in the $\Lambda$-convex hull of the constitutive set gives rise to exact solutions, the set of which in fact is residual in the Baire sense in a suitable function space of subsolutions. Theorem~\ref{prop:subsolutioncondition} states this `fatness' of the set of solutions in the whole state space, including density, while Theorem~\ref{prop:subsolutionconditiondens} allows the density to be fixed throughout the iteration, so that the density of the obtained solutions coincides with that of the original subsolution. The former result may seem strange, as it leads to solutions with irregular densities, although no density oscillations were involved at all in the construction! In a brief appendix, we exhibit this phenomenon for a toy model which we hope the reader will find equally instructive and amusing.

Some of the ideas employed in this paper have already been exhibited in the dissertation of Simon Markfelder~\cite[Chapter 4]{mark}, who obtains results of a comparable type. However, he only allows for constant generalized pressure (which in~\cite{mark} is denoted $\bar e$, and in the current note $Q$), whereas here, the generalized pressure can be essentially any continuous positive function. It has also been observed by Markfelder that the generalized pressure equals the total energy up to a constant factor precisely in the case of a \emph{monoatomic} gas, i.e., when the pressure takes the form $p(\rho)=\rho^{\frac2d +1}$, where $d$ is the space dimension, so that in this case the convex integration scheme admits control over the energy profile. As far as we can tell, this is a sheer but lucky coincidence.

As indicated above, two steps are involved in the construction of solutions via convex integration: The construction of the subsolution, and the actual iteration scheme. This paper concerns only the latter. Thus, although we argue that our scheme essentially exhausts the possibilities of convex integration in the compressible realm, by no means do we claim that this work supersedes previous ones, because we do not treat the problem of constructing subsolutions for various applications. We remark however that our Theorem~\ref{prop:subsolutionconditiondens} is used in forthcoming work by Gallenm\"uller--Wiedemann on the generability of measure-valued solutions, so that the results of the current note are actually of some use.

\section{Setup and Qualitative Analysis of the Compressible Euler System}
	In this section we describe a relaxation of system~\eqref{eq:compressEuler}, introduce all the necessary geometrical concepts, and prepare the toolbox for constructing wild solutions as described above, see Section~\ref{sec: Main_Proof} for the precise statements of the main resuls.
	
	\medskip
	
	\noindent Let $S^d, S^d_0$ denote the sets of $d\times d$ symmetric and trace-free symmetric matrices, respectively.
	For $m\in \R^d$ let us write
	\begin{equation}
	m\circ m:=m\otimes m-\frac1d |m|^2I_d,
	\end{equation}
	so that $m\circ m$ is always trace-free.
	Following~\cite{CFKW}, we set equations~\eqref{eq:compressEuler} in \emph{Tartar's framework} by relaxing to a linear system with pointwise constraints, by means of the following straightforward lemma.
	\begin{lemma}\label{lemma:reformulation}
		The compressible Euler system~\eqref{eq:compressEuler}
		is equivalent, in the sense of distributions, to the following linear system
		\begin{equation}\label{eq:relaxedsystem}
		\begin{aligned}
		\partial_t m + \diverg_x M + \nabla_x Q &= 0,\\
		\partial_t \rho + \diverg_x m &= 0,
		\end{aligned}
		\end{equation}
		with the pointwise constraints
		\begin{equation}\label{eq:constraints}
		M = \frac{m\circ m}{\rho},\;\;\;\; Q = p(\rho) + \frac1d\frac{|m|^2}{\rho}.
		\end{equation}
	\end{lemma}
	
	\noindent Hence, the state space now consists of the variables $z=(\rho, m, M, Q)\in \R^{N}$ with $N=\left(1+\frac d2\right)(d+1)$ (recall that $M$ is symmetric and trace-free) and $\R^N\simeq\R\times\R^d\times S_0^d\times\R$. Set
	\begin{equation}
	K=\left\{(\rho, m, M, Q)\in \R^{N}:\quad \rho > 0,\quad M=\frac{m\circ m}{\rho},\quad Q=p(\rho)+\frac{|m|^2}{\rho d}\right\}.
	\end{equation}

	\noindent Recall that the wave cone $\Lambda$ for~\eqref{eq:relaxedsystem} is the set of those non-zero $\bar z\in\R^N$ for which there exists a direction $\xi\in\R^{d+1}\backslash\{0\}$ such that $ z(t,x) =\bar z h((t,x)\cdot\xi)$ is a solution to~\eqref{eq:relaxedsystem}, for any choice of profile function $h:\R\to\R$. It was shown in~\cite{CFKW} that the wave cone for~\eqref{eq:relaxedsystem}  can be characterized as the set of $z=(\rho, m, M, Q)$ for which
	\begin{equation}
	\det\begin{pmatrix}
	m & M+QI_d\\
	\rho & m
	\end{pmatrix}=0.
	\end{equation}
	
	\noindent We shall now study the $\Lambda$-convex hull of $K$ with respect to this cone, and see that it is ``large'' enough so that we can obtain many solutions of the linear system~\eqref{eq:relaxedsystem}.
	For this, we first write 
	\begin{equation}
	e_{\kin}(\rho,m,M):=\frac d2\lambda_{max}\left(\frac{m\otimes m}{\rho}-M\right),
	\end{equation}
	similarly to the notation in~\cite{admissible}, where $\lambda_{max}(A)$ denotes the largest eigenvalue of $A$. The following lemma gathers some crucial properties of this ``generalised kinetic energy density''. It is a straightforward and minor modification of Lemma 3.2 in~\cite{admissible}.
	
	\begin{lemma}\label{lemma:ekinProperties}
		We have the following properties:
		\begin{enumerate}
			\item $e_{\kin}:\R^+\times  \R^d\times S^d_0\to \R$ is convex.
			\item $\frac{1}{2}\frac{|m|^2}{\rho}\le e_{\kin}(\rho,m,M)$, with equality if and only if $M=\frac{m\otimes m}{\rho}-\frac{1}{d}\frac{|m|^2}{\rho}I_d$.
			\item $|M|_\infty\le 2\frac{d-1}{d}e_{\kin}(\rho,m,M)$, where $|M|_\infty$ denotes the operator norm of the matrix.
			\item Let $\rho > 0$ be fixed. Then the convex hull of the set 
			\begin{equation}
			L_{\rho,r}\coloneqq\left\{ (m,M)\in\R^d\times S^d_0\;\colon\;\;\; |m|=r,\;\;\;(\rho,m,M,p(\rho)+\frac{|m|^2}{\rho d})\in K \right\}
			\end{equation}
			is given by
			\begin{equation}
			\left\{ (m,M) \in\R^d\times S^d_0\;\colon\;\;\; e_{\kin}(\rho,m,M)\le \frac{r^2}{2\rho}\right\}.
			\end{equation}
			In particular, the smallest value of $r$ such that $(m,M)\in (L_{\rho,r})^{co}$ is given by $\sqrt{2\rho e_{\kin} (\rho,m,M)}$.
		\end{enumerate}
	\end{lemma}

	\subsection{The Geometry}

	We shall now construct ``admissible line segments'' in the momentum and density directions. These will be used later to split the plane-wave solutions of the relaxed system into the conventional constant pressure directions, where the density stays the same, and suitable density directions. 
	Firstly, we observe that, for fixed $\rho$, we can oscillate the value of the momentum variable. Recall that $L_{\rho,r}$ denotes the slice of $K_r$ with fixed density.
	
	\begin{lemma}\label{lemma:admissiblesegments}
		Let $z=(\rho, m, M, Q)$ such that $(m,M)\in\interior(L_{\rho,r})^{co}$. There exist $\bar m\in\R^d$ and $\bar M\in S_0^d$ such that the ``momentum oscillation segment''
		\begin{equation}
		\sigma_m=\left[(\rho, m - \bar m, M - \bar M, Q), (\rho, m + \bar m, M + \bar M, Q)\right]
		\end{equation}
		satisfies $\left[(m - \bar m, M - \bar M), (m + \bar m, M + \bar M)\right]\in \interior(L_{\rho,r})^{co}$, and is parallel to the vector
		\begin{equation}
		\left(0, a, \frac{a\otimes a}{\rho}, 0\right) -\left(0, b, \frac{b\otimes b}{\rho}, 0\right), 
		\end{equation}   
		for some $a,b\in \R^d$ with $|a|=|b|=r$ and $b\neq \pm a$.
		Moreover 
		\begin{equation}
		|\bar m|\ge \frac{C}{r}(r^2 - |m|^2).
		\end{equation}
	\end{lemma}
	
	\begin{proof}
		We can adapt the geometric construction from~\cite[Lemma~6]{admissible}:
		 Indeed, the point $\underline{z}=(m,M)$ can be expressed as a finite convex combination of elements of $L_{\rho,r}$, i.e.,
		\begin{equation}
		\underline{z} = \sum\limits_{i=1}^{N+1}\lambda_i \underline{z}_i,
		\end{equation}
		where $\lambda_i\in(0,1)$, $\sum\lambda_i=1$, 
		\begin{equation}
		\underline{z}_i = \left(m_i,\frac{m_i\circ m_i}{\rho}\right),
		\end{equation}
		and $|m_i|=r$. By a slight perturbation, if necessary, we can ensure $m_i\neq -m_j$ for all $i,j\leq N+1$. Furthermore, we assume that the coefficients are ordered, so that $\lambda_1 = \max\limits_{1\leq i\leq N+1}\lambda_i$. 
		Since
		\begin{equation}
		\underline{z} - \underline{z}_1 = \sum\limits_{i=2}^{N+1}\lambda_i(\underline{z}_i-\underline{z}_1),
		\end{equation}
		we have
		\begin{equation}
		|m-m_1| \leq N \max_{2\leq i\leq N+1}\lambda_i|m_i-m_1|.
		\end{equation}
		Let $j>1$ be such that $\lambda_j|m_j-m_1| = \max\limits_{2\leq i\leq N+1}\lambda_i|m_i-m_1|$, and let
		\begin{equation}
		(\bar m,\bar M) = \frac12\lambda_j(\underline{z}_j-\underline{z}_1) = \frac12\lambda_j(m_j-m_1,\frac{m_j\circ m_j}{\rho}-\frac{m_1\circ m_1}{\rho}).
		\end{equation}
		Observe that, since $|m_i| = r$ for each $i$, we have
		\begin{equation}
		\frac{m_j\circ m_j}{\rho}-\frac{m_1\circ m_1}{\rho} = \frac{m_j\otimes m_j}{\rho}-\frac{m_1\otimes m_1}{\rho}.
		\end{equation}
		Now the line segment $\sigma_m$, defined by
		\begin{equation}
		\sigma_m\coloneqq \left[(\rho,\underline{z} - \frac{\lambda_j}{2}(\underline{z}_j-\underline{z}_1),Q),(\rho, \underline{z} + \frac{\lambda_j}{2}(\underline{z}_j-\underline{z}_1),Q)\right]
		\end{equation}
		satisfies the desired properties.
		Moreover, by construction we have the estimate
		\begin{equation}
		|\bar m| \geq \frac{1}{2N}|m-m_1| \geq \frac{1}{2N}(r-|m|) \geq \frac{1}{2N}\frac{(r-|m|)(r+|m|)}{2r} = \frac{1}{4rN}(r^2-|m|^2).
		\end{equation}
	\end{proof}

	We now prove the following characterization of the 
	$\Lambda$-convex hull of $K$ with respect to the wave cone $\Lambda$. This characterization is not explicitly used in the sequel, but it helps motivate the definition of subsolution later on, and is interesting as it shows that the results of Section~\ref{sec: Main_Proof} are in a sense optimal (cf.\ the introduction).
	
	\begin{theorem}
		\label{prop:LambdaConvexHull}
		We have the following identity:
		\begin{equation}
		K^\Lambda=\left\{z\in\R^N: \quad\rho>0,\quad M\in S_0^d, \quad Q\geq p(\rho)+\frac 2d e_{kin}(\rho,m,M)\right\}.
		\end{equation}
	\end{theorem}

	\begin{proof}
	Let us denote the function
	\begin{equation}
	F:\R^+\times\R^d\times S_0^d\times\R\ni z=(\rho,m,M,Q)\mapsto  p(\rho)+\frac 2d e_{kin}(\rho,m,M)-Q, 
	\end{equation}
	which clearly is convex. Since $F\restriction_K=0$, then by virtue of convexity and the inclusion $K^\Lambda\subset K^{co}$, $F\restriction_{K^\Lambda}\leq0$, which yields the first inclusion of the statement.

		For the other inclusion, consider first a state for which $F<0$, i.e., $Q > p(\rho)+\frac 2d e_{kin}(\rho,m,M)$. Then, there exists $r>0$ such that $Q=p(\rho)+\frac1d\frac{r^2}{\rho}$ and $r>\sqrt{2\rho e_{kin}(\rho,m,M)}$. According to Lemma~\ref{lemma:ekinProperties}, $(m,M)\in(L_{\rho, r})^{co}$, so that by (the proof of) Lemma~\ref{lemma:admissiblesegments} there exist vectors $m_1, m_2$ with $|m_1|=|m_2|=r$ such that $(\rho,m,M,Q)$ is a convex combination of two states whose difference $(\rho,m_1,m_1\circ m_1,Q)-(\rho,m_2,m_2\circ m_2,Q)$ (with $|m_1|=|m_2|$) has zero determinant and thus lies in the wave cone. This shows 
		\begin{equation}
		K^\Lambda\supset\{F<0\}.
		\end{equation}
		But since $K^\Lambda$ is closed (relatively in $\R^+\times\R^{N-1}$), it must even contain the closure of $\{F<0\}$, which is precisely $\{F\leq 0\}$ as $F$ is convex and $\{F<0\}\neq \emptyset$. 
	\end{proof}

	\subsection{Functional Setup}
	
We will construct our wild solutions to the Euler system upon perturbing subsolutions. Below we define the functional analytic framework in which we shall work as well as all the necessary notions.
	
	\begin{definition}
		Let $P(\rho)$ be the pressure potential defined by
		\begin{equation}
		P(\rho):= \rho \int^\rho_0\frac{p(r)}{r^2}\d r. 
		\end{equation}
	\end{definition}
	Clearly if $p(\rho)$ is strictly convex, then so is $P(\rho)$. 
	The total energy density for the Euler system is given by the sum of the kinetic energy density and the pressure potential
	\begin{equation}
	\frac{1}{2}\frac{|m|^2}{\rho}+P(\rho).
	\end{equation}
	
	For simplicity we shall consider a pressure law of the form $p(\rho):= \rho^\gamma$, for some $\gamma>1$. Then $P(\rho)= \frac{1}{\gamma-1}\rho^\gamma$.

	\medskip

	Let $\Omega\subset\R^d$ be open and let the energy profile
	\begin{equation}
	\bar e\in C((0,T)\times\bar\Omega) \cap C([0,T];L^1{\cap L^2}(\Omega))
	\end{equation}
	be given.

	\begin{definition}[Subsolutions]
		A subsolution to~\eqref{eq:compressEuler} with respect to the generalized pressure $Q\in C((0,T)\times\bar\Omega) \cap C([0,T];L^1{\cap L^2}(\Omega))$ is a quadruple $(\rho,m,M,Q)$ on $[0,T]\times\R^d$ into $\R^+\times\R^d\times S_0^d\times\R^+$ such that
		$$
		\rho\in C([0,T];L_w^\gamma(\R^d)),\;\; m\in C([0,T];L_w^2(\R^d)),\;\; M\in L^1((0,T)\times \Omega;S_0^d),
		$$
		and 
		\begin{equation}\label{eq:linearinmatrixform}
		\diverg_{(t,x)}
		\begin{pmatrix}
		m & M + QI_d \\
		\rho & m
		\end{pmatrix}
		=0,
		\end{equation}
		in the sense of distributions with 
		\begin{equation}
		\begin{aligned}
				Q &\geq p(\rho) + \tfrac 2d e_{\kin}(\rho,m,M)
		\end{aligned}
		\end{equation}
		on $(0,T)\times\Omega$.
	\end{definition}

	Now assume there exists a \emph{strict} subsolution $z_0 = (\rho_0, m_0, M_0, Q_0)$, that is: It satisfies~\eqref{eq:linearinmatrixform} in the sense of distributions in $(0,T)\times\R^d$, and
	\begin{itemize}
		\item $z_0(t,x)\in K$ for almost every $x\notin\Omega$;
		\item $z_0$ is continuous on $(0,T)\times\Omega$;
		\item  $Q_0 > p(\rho_0) + \tfrac 2d e_{\kin}(\rho_0,m_0,M_0)$ \; on $(0,T)\times\Omega$.
	\end{itemize}

	\begin{definition}[Approximating space of subsolutions]
		For $Q$ given as above and $z_0$ a bounded strict subsolution, let $X_0$ be the space of those 
		\begin{equation}
		(\rho,m)\in C([0,T];L_w^\gamma(\R^d))\times C([0,T];L_w^2(\R^d))
		\end{equation}
		{with $\supp{m(t,\cdot)}\subset\subset\Omega$ for every $t\in[0,T]$ and}
		for which $(\rho,m)(t,x)=(\rho_0,m_0)(t,x)$ for $t=0$ and $t=T$ for almost every $x\in\R^d$, and there is $M\in L^1((0,T)\times\Omega;S_0^d)$ such that $(\rho,m,M,Q)$ is a subsolution which agrees with $z_0$ almost everywhere outside $\Omega$, is continuous on $[0,T]\times\Omega$, and satisfies
		\begin{equation}
		 Q > p(\rho) + \tfrac 2d e_{\kin}(\rho,m,M)>p(\rho)+\frac{|m|^2}{d\rho}  \;\;\;\text{on}\;\;\;(0,T)\times\Omega.
		\end{equation}		
	\end{definition}
	
	\begin{prop}
		Let $X$ be the closure of $X_0$ with respect to the topology of $C([0,T];L_w^\gamma(\R^d))\times C([0,T];L_w^2(\R^d))$. Let $d_X$ denote a metric inducing this topology on $X$. Then $(X,d_X)$ is a complete metric space.
	\end{prop}

	\begin{proof}
		Using the definition of a subsolution and assumed regularity of the generalized pressure $Q$, we can see that for any $(\rho,m)\in X_0$ we have
		\begin{equation}
			\int_{\R^d} \rho^\gamma(t,x) \d x \leq \int_{\R^d\setminus\Omega} \rho_0^\gamma \d x + \int_{\Omega} Q(t,x) \d x, 
		\end{equation}
		and 
		\begin{equation}
			\begin{aligned}
			\int_{\R^d} |m(t,x)|^2 \d x = \int_{\Omega} \rho d \frac{|m(t,x)|^2}{\rho d} \leq 				d \int_{\{\rho\leq1\}} Q(t,x) \d x + d \int_{\{\rho>1\}} Q^2 \d x,
			\end{aligned}
		\end{equation}		
		for every $t\in[0,T]$. 
		Therefore $X_0$ comprises maps $(\rho,m)$ from the time interval $[0,T]$ into a bounded subset, say $B$, of $L^\gamma_w(\R^d)\times L^2_w(\R^d)$, which we can assume to be weakly closed without loss of generality. Therefore the weak topology on $B$ is metrizable, giving rise to a comlete metric space $(B,d_B)$. The space $X$ can be seen as the closure of $X_0$ in the complete metric space $C([0,T];B)$ with the metric $d_X(f_1,f_2) = \sup_{t\in[0,T]}\;d_B(f_1(t),f_2(t))$. Thus $(X,d_X)$ is itself a complete metric space. 
	\end{proof}

	We now define functionals designed to measure how much a subsolution differs from being an actual solution of the Euler equations. 
	
	\begin{definition}
		\label{def:EnergyFunctionals}
		For $\eps>0$ and any open bounded set $\Omega_0\subset \Omega$ we define the functional
		\begin{equation}
		\begin{aligned}
				I^{\eps ,\Omega_0}(\rho,m)&:= \inf_{t\in [\eps, T-\eps ]}\int_{\Omega_0} \left[ p(\rho)+ \frac{|m|^2}{d\rho}-Q  \right]  \d x.
		\end{aligned}    
		\end{equation}
		
	\end{definition}
	\noindent Clearly the functional defined above is bounded from below on $X$. Further, we have the following two important properties.
	
	\begin{lemma}\label{lemma:lsc}
		The functional $I^{\eps ,\Omega_0}$  is lower-semicontinuous on $X$.
	\end{lemma}
	
	\begin{proof}
		
		Suppose to the contrary that there are $(\rho_k,m_k)$ and $(\rho,m)$ such that $(\rho_k,m_k)\overset{d_X}{\to} (\rho,m)$, but $\lim\limits_{k\to\infty} I^{\eps ,\Omega_0}(\rho_k,m_k)< I^{\eps ,\Omega_0}(\rho,m)$.
		Thus 
		\begin{equation}
		\lim_{k\to\infty}\; \inf_{t\in [\eps, T-\eps ]}\int_{\Omega_0} \left[ p(\rho_k)+ \frac{|m_k|^2}{d\rho_k}-Q  \right]  \d x
		<
		\inf_{t\in [\eps, T-\eps ]}\int_{\Omega_0} \left[ p(\rho)+ \frac{|m|^2}{d\rho}-Q   \right]  \d x,  
		\end{equation}
		and so we can find a sequence of times $t_k\to t_0$ such that
		\begin{equation}
		\lim_{k\to\infty} \int_{\Omega_0} \left[ p(\rho_k(t_k,x))+ \frac{|m_k(t_k,x)|^2}{d\rho_k(t_k,x)}-Q  \right]  \d x
		<
		\int_{\Omega_0} \left[ p(\rho(t_0,x))+ \frac{|m(t_0,x)|^2}{d\rho(t_0,x)}-Q   \right]  \d x.  
		\end{equation}
		Since we have convergence in $X$, we also have
		\begin{equation}
		\begin{aligned}
		\rho_k(t_k,\cdot) &\rightharpoonup \rho(t_0,\cdot) &&\mathrm{in}\quad L^\gamma(\R^d)\quad \mathrm{weakly},\\ 
		m_k(t_k,\cdot) &\rightharpoonup m(t_0,\cdot) &&\mathrm{in}\quad L^2(\R^d)\quad \mathrm{weakly}.
		\end{aligned}
		\end{equation}
		Hence, by convexity
		\begin{equation}
		\liminf_{k\to \infty}  \int_{\Omega_0} \left[ p(\rho_k(t_k,x))+ \frac{|m_k(t_k,x)|^2}{d\rho_k(t_k,x)}-Q  \right]  \d x \ge \\ \int_{\Omega_0} \left[ p(\rho_k(t_0,x))+ \frac{|m_k(t_0,x)|^2}{d\rho_k(t_0,x)}-Q   \right]  \d x.
		\end{equation}
		This contradiction completes the proof.
	\end{proof}
		
	\begin{lemma}\label{lemma:I=0}
		We have that $I^{\eps,\Omega_0} \le 0$ on $X$. 
		If $I^{\eps,\Omega_0}(\rho,m)=0$ for every $\eps>0$ and every bounded open $\Omega_0\subset \Omega$, then $(\rho,m)$ is a weak solution of the compressible Euler equations~\eqref{eq:compressEuler} such that
		\begin{equation}\label{eq:regularity}
		\rho\in C([0,T];L_w^\gamma(\R^d)), \quad  m\in C([0,T];L_w^2(\R^d)),
		\end{equation}
		and 
		\begin{align}\label{eq:timeendpoints}
		&(\rho,m)(t,x) = (\rho_0,m_0)(t,x)\;\;\;\text{for}\;\;\;t=0,T\;\;\;\text{and\;\;a.e.}\;\;x\in\R^d,\\[1em]\label{eq:eKinonOmega}
		& p(\rho(t,x)) +\frac{|m(t,x)|^2}{\rho(t,x)d}=Q (t,x)
		\;\;\;\text{for every}\;\;\; t\in (0,T)\;\;\;\text{and\;\; a.e.}\;\;x\in \Omega.\\[1em]
				\end{align}
	\end{lemma}
	
	\begin{proof}
		Let $(\rho,m)\in X_0$. Then there is a subsolution $(\rho,m,M,Q)$ such that 
				\begin{equation}
		Q  > p(\rho)+\frac{2}{d}e_{\kin}(\rho,m,M)> p(\rho )+\frac{|m|^2}{d\rho}
		\end{equation}
		on $(0,T)\times\Omega$. 
		Hence $I^{\eps,\Omega_0} < 0$ on $X_0$, and by the definition of $X$ and lower-semicontinuity we have $I^{\eps,\Omega_0} \le 0$ on $X$.
		
		Now suppose that $I^{\eps,\Omega_0}(\rho,m)=0$ for some $(\rho,m)\in X$. Then ~\eqref{eq:regularity},~\eqref{eq:timeendpoints}, and~\eqref{eq:eKinonOmega} are satisfied due to the definition of $X$.
		
		Choose $(\rho_k,m_k)\in X_0$ such that 
		$(\rho_k,m_k) \overset{d_X}{\to}(\rho,m)$ and let $(M_k)$ be the associated sequence of matrix fields $M_k$. We then have, by Lemma~\ref{lemma:ekinProperties}, 
		\begin{equation}
		|M_k|_\infty < \frac{2(d-1)}{d}e_{\kin}(\rho_k,m_k,M_k) <(d-1)\bar e 
		\end{equation}
		in $[0,T]\times\Omega$, while $M_k=M_0$ for $x$ outside $\Omega$.
		Therefore we can assume that $M_k\overset{*}{\rightharpoonup}M$ in $L^\infty((0,T)\times\R^d)$. 
		Furthermore, we have
		\begin{equation}
		\begin{aligned}
		\rho_k &\to \rho &&\mathrm{in}\quad C([0,T];L_w^\gamma(\R^d)),\\
		m_k &\to m &&\mathrm{in}\quad C([0,T];L_w^2(\R^d)),
		\end{aligned}
		\end{equation}
		so that the linear equations 
		\begin{align}
		\partial_t m + \diverg_x M +\nabla_x Q &= 0,\\ 
		\partial_t\rho +\diverg_x m &=0
		\end{align}
		hold in the sense of distributions. On $(0,T)\times\Omega$ we have for each $k$
		\begin{equation}
		\begin{aligned}
				p(\rho_k) + \frac2d e_{\kin}(\rho_k,m_k,M_k) < Q.
		\end{aligned}
		\end{equation}
		Therefore, since $p$ and $e_{\kin}$ are convex, in the limit we have
		\begin{equation}
		\begin{aligned}
				p(\rho) + \frac2d e_{\kin}(\rho,m,M) \leq Q.
		\end{aligned}
		\end{equation}
		Therefore, using~\eqref{eq:eKinonOmega}, we have
		\begin{equation}
		e_{\kin}(\rho,m,M) \leq \frac{|m|^2}{2\rho},
		\end{equation}
		which implies, by Lemma~\ref{lemma:ekinProperties}, that $e_{\kin}(\rho,m,M) = \frac{|m|^2}{2\rho}$, and so
		\begin{equation}
		M = \frac{m\circ m}{\rho}.
		\end{equation}
		Consequently, using~\eqref{eq:eKinonOmega} we conclude also
		\begin{equation}
		Q = p(\rho) + \frac{|m|^2}{\rho d}.
		\end{equation}
		Note that outside of $\Omega$, the last two equalities are satisfied trivially, since $(\rho,m,M,Q)$ agrees there with $z_0$.
		Therefore, by virtue of Lemma~\ref{lemma:reformulation}, it follows that $(\rho,m)$ is a weak solution of~\eqref{eq:compressEuler} on $(0,T)\times\R^d$.
	\end{proof}

	\subsection{The ``Gluing'' Procedure}

	The strategy for constructing solutions to the Euler system is to start with a strict subsolution, for which the values of the ``energy functionals'' defined in Definition~\ref{def:EnergyFunctionals} are strictly negative, and add highly oscillatory perturbations to increase their value. Iterating this procedure will eventually lead to an element of the space $X$ whose value of the functional $I$ is zero, i.e., a solution to the Euler equations.
	Importantly, we need to guarantee that, at each step, the perturbed subsolution will stay in the space $X_0$. To do so we will want to use plane-wave solutions to \eqref{eq:relaxedsystem} as perturbations.
	However, different perturbations will be needed, depending on space and time. Therefore we want some form of local wave solutions. 
	To generate local perturbations we will take wave solutions to the relaxed linear system~\eqref{eq:relaxedsystem} and ``glue'' them together. To localise a function smoothly we could naively just multiply a plane-wave solution by a smooth cut-off function. However, such a product would no longer be a solution to~\eqref{eq:relaxedsystem}. 
	Instead, we can use a potential operator for the relaxed linear system, so that if we cut-off the candidate function for our solution and then apply the potential operator, this will create a smooth localised solution to \eqref{eq:relaxedsystem}. In what follows we first show that such a potential operator exists, and then prove that the error made by applying the cut-off in the potential space can be controlled.
	
	We recall the following tool from~\cite[Proposition 4]{admissible}:
	
	\begin{proposition}[Potential for constant density] 
		\label{prop:potentialsdens}
		There exists a homogeneous linear differential operator 
		\begin{equation}
		{\mathcal{L}_{dens}}: C^\infty((0,T)\times\R^d)\to C^\infty((0,T)\times \R^d; \R^d\times S_0^d)
		\end{equation}
		such that the following holds: ${\mathcal{L}_{dens}}\Phi$ satisfies
		\begin{equation}\label{relaxedconstdens}
		\diverg_{(t,x)} \begin{pmatrix}
		m & M \\
		0 & m
		\end{pmatrix}
		=0
		\end{equation}
		for each smooth $\Phi\in C^\infty((0,T)\times\R^d)$, and for every plane wave solution $U$ of~\eqref{eq:relaxedsystem} there exists $\Phi$ such that $ {\mathcal{L}_{dens}}\Phi=U$. 
	
	\end{proposition}

Note that the following lemma applies in particular to $\mathcal L_{dens}$. 
	
	\begin{lemma}\label{lem: difference}
		Let $\mathcal{L}$ be an $l$-homogeneous linear differential operator with constant coefficients and let $\Phi, U$ be smooth periodic functions such that $\mathcal{L} \Phi =U$. Then, for any $n\in\N$ and any smooth scalar function $\chi$, we have
		\begin{equation}\label{eq:Linfinityerrorcontrol}
		\bigg\| \chi U(nx)- \mathcal L\left( \chi \frac{1}{n^l}\Phi(nx)\right ) \bigg\|_{L^\infty} 
		\le C(\mathcal L,  \|\Phi\|_{C^l},  \|\chi\|_{C^l} ) \frac{1}{n}.
		\end{equation}
	\end{lemma}
	\begin{proof}
		We see from definition of $l$-homogeneity and an application of the chain rule that
		\begin{align}
		\bigg\| \chi U(nx)&- \mathcal L\left( \chi \frac{1}{n^l}\Phi(nx)\right ) \bigg\|_{L^\infty} 
		= \left\|\chi \mathcal L\left(  \frac{1}{n^l}\Phi(nx)\right )- \mathcal L\left(\chi  \frac{1}{n^l}\Phi(nx)\right )  \right\|_{L^\infty}
		\\
		&=\left\| \sum_{|\alpha|+|\beta|=l, |\beta|\neq 0} L^{\alpha,\beta}\frac{1}{n^{l-|\alpha|}} \partial_{\beta} \chi \partial_\alpha  \Phi(n x ) \right\|_{L^\infty}\le C(\mathcal L, \|\Phi\|_{C^l},  \|\chi\|_{C^l} ) \frac{1}{n},
		\end{align}
		where $L^{\alpha,\beta}$ are the appropriate coefficients depending (only) on the operator $\mathcal L$.
		\end{proof}

	\section{Perturbation Property}\label{sec:perturb}

	The cornerstone of the convex integration method, and the main ingredient towards the proof of our main result, is the following perturbation property.
	
	\begin{proposition}[Perturbation property]\label{prop:PertuabationProperty}
		Let $\Omega_0\subset\Omega$ and $\eps >0$ be given. For all $\alpha >0$ there exists $\beta>0$, possibly depending on $\eps$ and $\Omega_0$, such that, whenever $(\rho,m)\in X_0$ with 
		\begin{equation}
		I^{\eps,\Omega_0}(\rho,m)<-\alpha, 
		\end{equation}
		then there exists a sequence $\{\rho_k,m_k\}\subset X_0$ with $(\rho_k,m_k)\overset{d_X}{\to}(\rho,m)$
		and 
		\begin{equation}
		\liminf_{k\to \infty} I^{\eps ,\Omega_0}(\rho_k,m_k) \ge I^{\eps ,\Omega_0}(\rho,m) +\beta.
		\end{equation}
	\end{proposition}
	
	The remainder of this section is dedicated to proving this result. The proof is rather lengthy, so for clarity it is split into several subsections.

	\subsection{The Form of the Perturbations}\label{pertform}
	
	For a point $z=(\rho,m,M,Q)\in\R^N$ let us write
	\begin{equation}
	U:=
	\begin{pmatrix}
	m & M + QI_d \\
	\rho  & m
	\end{pmatrix},
	\end{equation}
	so that system \eqref{eq:relaxedsystem} can be written as
	\begin{equation}
	\diverg_{(t,x)}U=0.
	\end{equation}
	
	As mentioned before, we will seek to use plane-wave solutions of the form 
	\begin{equation}
	\bar z h(n [(t,x)\cdot \xi] ),
	\end{equation}
	where $n\in\N$ is the iteration step, $\xi\neq0$ is a direction in $\R\times\R^d$, $\bar z\in\R^N$, and the profile $h:\R\to\R$ will be chosen later to be a periodic function of small period (thus making the wave oscillate with high frequency) and zero mean.
	For this ansatz to solve the equation we need
	\begin{equation}
	\diverg_{(t,x)}\left(\bar z h(n[ (t,x)\cdot \xi])\right)=n h'\left(n[ (t,x)\cdot \xi] \right)\bar z\xi =0,
	\end{equation}
	and so we must have $\bar z \xi =0$. If $\bar\rho=\bar Q=0$, using Proposition \ref{prop:potentialsdens} we know that there exists a function $\Phi$ such that $\mathcal L_{dens}\Phi=\bar z h(n [(t,x)\cdot \xi] )$. 

	Let us now describe the particular form of perturbations in momentum which we will use below to obtain the perturbation property.
	We will naturally oscillate in the manner described before in the construction of admissible oscillation segments. Thus, we have $\bar z$ of the form 
	
	\begin{equation}
	\bar z = (0,\bar m, \bar M, 0),
	\end{equation}
	so that the corresponding matrix $U$ has the form
	\begin{equation}
	\begin{pmatrix}
	a-b & a\otimes a -b \otimes b \\
	0 &a-b 
	\end{pmatrix},
	\end{equation}
	for some $a,b\in \R^d$ with $|a|=|b|$ and $a\neq \pm b$. For $d\ge 2$ this matrix has zero determinant and so there must exist a non-zero $\xi\in\R^{d+1}$ such that $ \bar z\xi =0$.  
	Further, notice that because $a\neq \pm b$, $\xi$ is not parallel to the time direction, $e_t$. Indeed, if it were, then $\bar z \xi $ would be a scalar multiple of $(a_1-b_1,a_2-b_2,\cdots, a_d-b_d,0)$, which is non-zero. 
	
	As a profile function for these oscillations we will take
	\begin{equation}
	\bar h(x)=\begin{cases}
	-1 & \mathrm{for} \quad x\in [0,\frac{1}{2}),\\
	1 & \mathrm{for} \quad x\in [\frac{1}{2},1),
	\end{cases}
	\end{equation}
	and then define $h$ by taking the periodic extension of $\bar h$, and mollify with a very small parameter $\eps_0$ that will go to zero with the oscillation frequency; we ignore $\eps_0$ in the sequel. 
	This choice is smooth and has mean zero, giving equal weight to $m-\bar m$ and $m+\bar m$.

	\medskip
	
	Thus we obtain the desired perturbations of the form
	\begin{equation}
	\bar z h(n[(t,x)\cdot \xi]).
	\end{equation}
	
	We will now prove a straightforward weak convergence lemma that will play a role similar to Lemma $7$ in \cite{admissible}.
	
	\begin{lemma}\label{YMgeneration}
		Let $h\in L^\infty(\R;\R)$ be a bounded $L$-periodic function and $\xi\in \R^{d+1}$ be a vector not parallel to the time direction $e_t$. Then the sequence $\{h(n(t,x)\cdot \xi)\}_{n\in\N}$ generates the Young measure 
		\begin{equation}
		\nu=\frac 1L dx\restriction_{[0,L)}\circ h^{-1}
		\end{equation}
		uniformly in $t$, more precisely: For every $f\in C(\R;\R)$ and every $\phi\in L^1(\R^d;\R)$, we have
			\begin{equation}
			\lim_{n\to\infty}\sup_{t\in\R}\left|\int_{\R^d}\phi(x)f(h(n(t,x)\cdot \xi))dx-\int_{\R^d}\phi(x)dx\frac1L\int_0^L f(h(s))ds\right|=0.
			\end{equation}
	\end{lemma}
	
	\begin{proof}
	Since $f\circ h$ is bounded and $L$-periodic if $h$ is, we may assume without loss of generality $f$ to be the identity function. By assumption, we can write $\xi=(\xi_0,\xi')$ with $\xi'\neq 0$. For given $t\in\R$ and $n\in\N$, denote by $\alpha_n(t)$ the value of $\xi_0 t$ modulo $\frac Ln$, that is, 	
	\begin{equation}
	\alpha_n(t):= \min\left\{\xi_0 t-k\frac Ln: \quad k\in\N,\quad k\frac Ln\leq \xi_0t\right\}\in\left[0,\frac Ln\right).
	\end{equation}
	In particular, $\alpha_n(t)\to 0$ as $n\to\infty$, uniformly in $t\in\R$.
	
	Then, denoting $h_n:=h(n\cdot)$ and noting that $h_n$ is $\frac Ln$-periodic, we can compute
	\begin{equation}
	\begin{aligned}
	\int_{\R^d}\phi(x)h(n(t,x)\cdot \xi)dx&=\int_{\R^d}\phi(x)h_n(t\xi_0+x\cdot\xi')dx\\
	& =\int_{\R^d}\phi(x)h_n(\alpha_n(t)+x\cdot\xi')dx\\
	& =\int_{\R^d}\phi(x)h_n\left(\left(x+\alpha_n(t)\frac{\xi'}{|\xi'|^2}\right)\cdot\xi'\right)dx\\
	& =\int_{\R^d}\phi\left(x-\alpha_n(t)\frac{\xi'}{|\xi'|^2}\right)h_n(x\cdot\xi')dx.
	\end{aligned}
	\end{equation}
	It is well-known (see~\cite{ballmurat}) that the sequence $\{h_n(x\cdot\xi')\}_{n\in\N}$ generates the Young measure $\frac 1L dx\restriction_{0,L}\circ h^{-1}$. Moreover, $\phi\left(\cdot-\alpha_n(t)\frac{\xi'}{|\xi'|^2}\right)\to\phi$ in $L^1(\R^d)$, uniformly in $t$, whence the claim follows.

	\end{proof}

	\subsection{The Shifted Grid}
	
	On $ \R \times\R^d$ we define the grid of size $h$ exactly as in~\cite{admissible}:
	For $\zeta\in \Z^d$ and $|\zeta|:=\zeta_1+\cdots +\zeta_d$, let $Q_\zeta,\tilde Q_\zeta $ be cubes on $\R^d$ centered at $\zeta h$ with side length $h$ and $\tfrac{3}{4} h$, respectively, so that
	\begin{equation}
	Q_\zeta:= \zeta h+ \left[ -\tfrac{h}{2},\tfrac{h}{2}\right]^d,\quad \tilde Q_\zeta:= \zeta h+ \left[ -\tfrac{3h}{8},\tfrac{3h}{8}\right]^d.
	\end{equation}
	For every $(i,\zeta)\in \Z\times \Z^d$ let 
	\begin{equation}
	C_{i,\zeta}=\begin{cases}
	[ih,(i+1)h]\times Q_\zeta  & \text{if} \;\;|\zeta|\;\; \text{is even},\\
	[(i-\tfrac{1}{2})h,(i+\tfrac{1}{2})h] \times Q_\zeta  & \text{if} \;\;|\zeta|\;\; \text{is odd}.
	\end{cases}
	\end{equation}
	Let further $0\le \varphi\le 1$ be a smooth cutoff function on $\R \times\R^d$, with support in $[-\tfrac{h}{2},\tfrac{h}{2}]^{d+1}$, identically $1$ on $[-\tfrac{3h}{8},\tfrac{3h}{8}]^{d+1}$, and strictly smaller that 1 outside. Let $\varphi_{i,\zeta}$  be the translation of $\varphi$ to the associated $C_{i,\zeta}$ and let $\phi_h$ denote the sum of all $\varphi_{i,\zeta}$. 
	
	Given an open bounded $\Omega_0$, let 
	\begin{equation}
	\Omega_1^h=  \bigcup\; \left\{\tilde Q_\zeta\;\colon\;\;\; |\zeta|\; \text{even,}\;\;\; Q_\zeta \subset \Omega_0  \right\}, \quad \Omega_2^h= \bigcup\; \left\{\tilde Q_\zeta\;\colon\;\;\; |\zeta|\; \text{odd,}\;\;\; Q_\zeta \subset \Omega_0  \right\}.
	\end{equation}
	Note that, for $s=1,2$, 
	\begin{equation}
	\lim_{h\to 0} |\Omega^h_s| =\frac{1}{2}\left(\frac{3}{4}\right)^d |\Omega_0|, 
	\end{equation}
	and for every $t$ the set $\{x\in \Omega_0\;\colon\; \phi_h(t,x)=1\}$ contains at least one of the sets $\Omega_s^h$. See~\cite[Figure 1]{admissible} for illustration.
	
	Now let $(\rho,m)\in X_0$ with 
	\begin{equation}
	I^{\eps,\Omega_0}(\rho,m)\le -\alpha
	\end{equation}
	for some $\alpha>0$, and let $M$ be the associated smooth matrix field solving \eqref{eq:linearinmatrixform}. Let 
	\begin{equation}
	A=\max_{[\frac{\eps}{2},T-\frac{\eps }{2}]\times \Omega_0}Q,
	\end{equation}
	and  
	further let\, $E^h: [\eps ,T-\eps ]\times \Omega_0\to \R$ be a step-function defined on the grid by 
	\begin{align}
	E^h(t,x)=E^h(ih,\zeta h) =p(\rho(ih,\zeta h))+\frac{1}{d}\frac{|m(ih,\zeta h)|^2}{\rho(ih,\zeta h)}-Q(ih,\zeta h)\quad &\text{for} \, (t,x)\in C_{i,\zeta},
		\end{align}
	which is well-defined for $h\le \eps$. Since $\rho,m,$ and $Q$ are uniformly continuous on $[\tfrac{\eps}{2},T-\tfrac{\eps}{2}]\times \Omega_0$, for any $s=1,2$ we have
	\begin{align}
	\lim_{h\to 0} \int_{\Omega^h_s} E^h(t,x)\d x &=\frac{1}{2}\left(\frac{3}{4}\right)^d \int_{\Omega_0} p(\rho(ih,\zeta h))+\frac{1}{d}\frac{|m(ih,\zeta h)|^2}{\rho(ih,\zeta h)}-\bar e(ih,\zeta h)\d x, 
		\end{align}
	uniformly in $t\in[\eps ,T-\eps]$. In particular there exists a dimensional constant $c>0$ such that for sufficiently small $h$ and for any $t\in [\eps, T-\eps]$, we have 
	\begin{align}\label{eq:Functionalcontrols}
	\int_{\Omega^h_s} E^h(t,x)\d x\ge c \alpha\quad &\text{whenever}\quad   \int_{\Omega_0} p(\rho(ih,\zeta h))+\frac{1}{d}\frac{|m(ih,\zeta h)|^2}{\rho(ih,\zeta h)}-Q(ih,\zeta h)\d x\le -{\alpha}.
		\end{align}
	
	For each $(i,\zeta)\in \Z\times \Z^d$ such that $C_{i,\zeta}\subset[\tfrac{\eps}{2},T-\tfrac{\eps}{2}]\times\Omega_0$ let
	\begin{equation}
	z_{i,\zeta}=(\rho(ih,\zeta h),m(ih,\zeta h),M(ih,\zeta h),Q(ih,\zeta h) ),
	\end{equation}
	and choose the momentum admissible line segment in the following way: 

	We apply Lemma \ref{lemma:admissiblesegments} to $z=z_{i,\zeta}$ with 
	\begin{equation}
	r=\sqrt{d\rho(ih,\zeta h)Q(ih,\zeta h)};
	\end{equation}
	note that, with this choice, $z_{i,\zeta}\in \interior(L_{\rho(i,\zeta),r})^{co}$ owing to Lemma~\ref{lemma:ekinProperties} and the property of $z$ being a strict subsolution. This yields the momentum segment
	\begin{equation}
	\sigma^m_{i,\zeta}=[z_{i,\zeta}-\bar z_{i,\zeta}^m,z_{i,\zeta}+\bar z_{i,\zeta}^m].
	\end{equation}

	Since $ z:=(\rho, m, M,Q )$ is uniformly continuous, for sufficiently small $h$ we have that 
	\begin{align}\label{eq:continuitycontrols}
	&p(\rho(t,x))+\frac 2d e_{\kin}(z(t,x)+Z)< Q(t,x) \quad\text{for all }\; Z\in\sigma^m_{i,\zeta} \;\text{and} \; (t,x)\in C_{i,\zeta}.
	\end{align}
	Indeed, the estimate follows since, by construction, it is satisfied at the point $(t,x)=(ih, \zeta h)$, and then one uses continuity. 
	Thus, given $\alpha>0$, we fix the grid size $0<h<\frac{\eps}{2}$ so that the estimates \eqref{eq:Functionalcontrols} and \eqref{eq:continuitycontrols} hold.

	\subsection{Adding the Perturbations}
	
	For each $(i,\zeta)$ and $j,n\in \N$ we choose the vector $\bar z_{i,\zeta}$ as in the previous subsection, as well the corresponding wave direction $\xi_{i,\zeta}\in\R^{d+1}\setminus\{0\}$, not parallel to the time direction, such that $\bar z_{i,\zeta}h((t,x)\cdot\xi_{i,\zeta})$ satisfies~\eqref{eq:relaxedsystem} for any profile function $h$. Let also $\phi_{i,\zeta}$ be the cutoff function introduced in the previous subsection, and $h$ the profile from Subsection~\ref{pertform}. 
	
	{Then from Proposition~\ref{prop:potentialsdens} we infer existence of a smooth function $\Phi_{i,\zeta}\in C^\infty((0,T)\times\R^d)$ such that $\mathcal L_{dens}\Phi = \bar z_{i,\zeta}h((t,x)\cdot\xi_{i,\zeta})$. Now we define
	
	\begin{equation}
		(0,\tilde m^j_{i,\zeta},\tilde M^j_{i,\zeta},0)(t,x):= \mathcal{L}_{dens}\left(\frac{1}{j^l}\phi_{i,\zeta}(t,x)\Phi_{i,\zeta}(jt,jx)\right),
	\end{equation}
and	
	\begin{equation}
		(0,\bar m^j_{i,\zeta},\bar M^j_{i,\zeta},0)(t,x):= \phi_{i,\zeta}(t,x)\bar  z_{i,\zeta}h(j[(t,x)\cdot \xi_{i,\zeta}]).
	\end{equation}
Then Lemma~\ref{lem: difference} gives the following estimate
	\begin{equation}
		\norm{\bar m^j_{i,\zeta} - \tilde m^j_{i,\zeta}}_{L^\infty} \leq \frac{C}{j}.
	\end{equation}
}
	
	
	Due to our choice of the function $h$ and Lemma~\ref{YMgeneration}, we readily see that the sequence $\bar m^j_{i,\zeta}$ generates the following homogeneous Young measure on the set $\{\phi_{i,\zeta}=1\}$, uniformly in time:
	\begin{align}
		\nu^m_{i,\zeta}&\coloneqq 
	\frac{1}{2}\delta_{-\bar m_{i,\zeta}}+\frac{1}{2}\delta_{\bar m_{i,\zeta}}.
	\end{align}
	
	By Lemma~\ref{lem: difference}, the same Young measure is generated by the sequence $\tilde m^j_{i,\zeta}$.

	Then for the kinetic energy density we obtain, as $j\to \infty$,
	\begin{align}
	\frac{1}{2}\frac{|m+\bar m^j_{i,\zeta}|^2}{\rho}&\stackrel{*}{\rightharpoonup}
	\int_{\R^d}\frac{1}{2}\frac{|m+y|^2}{\rho}\d \nu^m_{i,\zeta}(y)\\
	&=
	\int_{\R^d}\frac{1}{2}\frac{|m|^2+|y|^2+2m\cdot y}{\rho}\d \nu^m_{i,\zeta}(y)
	\\[0.3em]
	&=
	\frac12\frac{|m|^2}{\rho}+\frac12\frac{|\bar m_{i,\zeta}|^2}{\rho}.
	\end{align}
	
	We now want to find a lower bound for the
	extra term $\frac{|\bar m_{i,\zeta}|^2}{2\rho}$ in terms of $| E^h(t,x)|$.

	For the momentum we have, by setting $r=\sqrt{d \rho Q}$,
	\begin{equation}
	|\bar m_{i,\zeta}|^2\ge\frac{C}{r^2}|r^2-|m|^2|^2=\frac{C}{d \rho Q}
	|d\rho Q-|m|^2|^2=\frac{C\rho}{ Q}
	\left|Q-\frac{|m|^2}{d\rho}\right|^2,
	\end{equation}
	and so
	\begin{equation}
	\frac{|\bar m_{i,\zeta}|^2}{2\rho} \ge\frac{C}{ Q}
	\left|Q-\frac{|m|^2}{d\rho}\right|^2.
	\end{equation}
	Now we can bound $Q$ by $A$ and obtain
	\begin{equation}
	\label{eq:mombound}
	\frac{|\bar m_{i,\zeta}|^2}{2\rho} \ge\frac{C}{ A}
	\left|E^h(ih,\zeta h)\right|^2.
	\end{equation}

	\medskip	
	
	As $\mathcal L$ is linear, we can sum up all the perturbations over $(i,\zeta)$ to get a solution in the whole space. For each $j\in\N$ we write 
	\begin{align}
		(0,\tilde m_{j},\tilde M_{j},0)&:= \sum_{\{(i,\zeta): C_{i,\zeta}\subset[\eps,T-\eps]\times \Omega_0  \}}(0,\tilde m_{i,\zeta},\tilde M_{i,\zeta},0)\\
	(0, \bar m_{j},\bar M_{j},0)&:= \sum_{\{(i,\zeta): C_{i,\zeta}\subset[\eps,T-\eps]\times  \Omega_0 \}}(0,\bar m_{i,\zeta},\bar M_{i,\zeta},0),
	\end{align} 
	and 
	\begin{align}
	( \rho_{j}, m_j, M_j,Q)&= ( \rho, m, M, Q)+(0, \tilde{m}_j,\tilde{M}_j,0).
	\end{align}

	The above sums have a finite number of terms and thus from \eqref{eq:continuitycontrols} and \eqref{eq:Linfinityerrorcontrol}
	we can find a $k$ large enough such that $(\rho_j,m_j)\in X_0$ for all $j>k$. 
	Using~\eqref{eq:mombound}, we see that 	
	\begin{align}
	\liminf_{j\to \infty}I^{\eps,\Omega_0}(\rho_j,m_j)=& \liminf_{j\to \infty} \inf_{t\in [\eps,T-\eps]} \left\{ \int_{\Omega_0} p(\rho)+\frac{|m+\tilde m_j|^2}{d\rho}-\bar e(t,x)\d x  \right\}
	\\
	&\ge I^{\eps,\Omega_0}+\inf_{t\in [\eps,T-\eps]} \Bigg\{ \int_{\Omega_0} \sum_{s=1,2} \int_{\Omega^h_s}\sum_{\{(i,\zeta)| C_{i,\zeta}\subset[\eps,T-\eps]\times\Omega_0 \}} \frac{|\bar m_{i,\zeta}|^2}{d\rho} \d x \Bigg\}
	\\
	&\ge I^{\eps,\Omega_0}+\inf_{t\in [\eps,T-\eps]} \left\{ \int_{\Omega_0} \sum_{s=1,2}
	\int_{\Omega^h_s}\frac{C}{A}|E^h(ih,\zeta h)|^2 \d x \right\}
	\\
	&\ge I^{\eps,\Omega_0}+\inf_{t\in [\eps,T-\eps]} \left\{ \int_{\Omega_0}\frac{C}{A|\Omega_0|}\sum_{s=1,2}
	\left(\int_{\Omega^h_s}|E^h| \d x\right)^2 \right\},
	\end{align}
	where for the last inequality we have used H\"older's inequality. Using \eqref{eq:Functionalcontrols} we see that 
	\begin{equation}
	\liminf_{j\to \infty}I^{\eps,\Omega_0}(\rho_j, m_j)\ge I^{\eps,\Omega_0}(\rho, m) +\frac{C'}{A|\Omega_0|}\alpha^2,
	\end{equation}
	and set $\beta =\tfrac{C'}{A|\Omega_0|}\alpha^2$.
	Furthermore, as for $j$ large enough $(\rho_j,m_j)\in X_0$ for all $j>k$, we deduce that $(\rho_j,m_j) \overset{d_X}{\to} (\rho,m)$ as $j\to \infty$. Thus the proof of Proposition~\ref{prop:PertuabationProperty} is complete.

	\section{Statement and Proof of the Main Results}
	\label{sec: Main_Proof} 
	
	Having established the perturbation property we are ready to state and prove our main results.

		\begin{theorem}\label{prop:subsolutioncondition}
		Let $\Omega\subset\R^d$ be open and let the generalized pressure
	\begin{equation}
	Q\in C((0,T)\times\bar\Omega) \cap C([0,T];L^1\cap L^2(\Omega))
	\end{equation}
	be given.
	Further let $(\rho_0,m_0,M_0,Q)$ be a continuous solution of~\eqref{eq:relaxedsystem} in $[0,T]\times\R^d$, such that $(\rho_0,m_0,M_0,Q)\in K$ for $x\nin\Omega$ and
		\begin{align}
		Q> p(\rho_0)+\frac 2d e_{\kin}(\rho_0,m_0,M_0)
		\end{align}
		in $(0,T)\times\Omega$.
		Then the set of weak solutions of~\eqref{eq:compressEuler} that agree with $(\rho_0,m_0)$ in $\R^d\times\{0,T\}\cup(\R^d\setminus\Omega)\times(0,T)$, and such that 
		\begin{equation}
		Q=p(\rho)+\frac{|m|^2}{\rho d}
		\end{equation}
		almost everywhere in $[0,T]\times\Omega$, is residual in the space $X$ in the Baire sense. The same is true when $\R^d$ is replaced by $\T^d$.
	\end{theorem}

As $\rho$ remains unchanged in the entire scheme, we also obtain the following more classical result: 

\begin{theorem}\label{prop:subsolutionconditiondens}
		Let $\Omega\subset\R^d$ be open and let the generalized pressure
	\begin{equation}
	Q\in C((0,T)\times\bar\Omega) \cap C([0,T];L^1\cap L^2(\Omega))
	\end{equation}
	be given.
	Further let $(\rho_0,m_0,M_0,Q_0)$ be a continuous solution of~\eqref{eq:relaxedsystem} in $[0,T]\times\R^d$, such that $(\rho_0,m_0,M_0,Q_0)\in K$ for $x\nin\Omega$ and
		\begin{align}
		\bar e = Q_0> p(\rho_0)+\frac 2d e_{\kin}(\rho_0,m_0,M_0)
		\end{align}
		in $(0,T)\times\Omega$.
		Then the set of weak solutions of~\eqref{eq:compressEuler} such that $m=m_0$ in $\R^d\times\{0,T\}\cup(\R^d\setminus\Omega)\times(0,T)$ and $\rho=\rho_0$ everywhere in $[0,T]\times\R^d$, and such that 
		\begin{equation}
		\bar e=Q_0=p(\rho_0)+\frac{|m|^2}{\rho_0 d}
		\end{equation}
		almost everywhere in $[0,T]\times\Omega$, is residual in the space $X\cap\{\rho=\rho_0\}$ in the Baire sense. The same is true when $\R^d$ is replaced by $\T^d$.
	\end{theorem}
	
We prove here only Theorem~\ref{prop:subsolutioncondition}, as Theorem~\ref{prop:subsolutionconditiondens} follows in an analogous way. Likewise, we omit the proof on $\T^d$, as it is completely analogous to the case of the whole space.

The proof follows the by now standard Baire category argument, as in~\cite{DLS2009}, with necessary modifications.
	
	\begin{proof}
		
		The functional $I^{\eps,\Omega_0}$ is lower-semicontinuous on the complete metric space $X$ and takes values in $\R$, and so it can be written as a pointwise supremum of countably many continuous functionals.  Therefore it is a Baire-$1$ map and hence its points of continuity from a residual set in $X$.
		
		We will now argue that $I^{\eps,\Omega_0}$ vanishes at its continuity points.
		Assume first, towards a contradiction, that there exists a $(\rho,m)\in X$ which is a point of continuity of $I^{\eps,\Omega_0}$ and $I^{\eps,\Omega_0}\le -\alpha$ for some $\alpha>0$. 
		
		 Choosing a sequence $\{(\rho_k,m_k)\}_k\subset X_0$ such that $(\rho_k,m_k)\overset{d_X}{\to}(\rho,m)$, we have $I^{\eps,\Omega_0}(\rho_k,m_k)\to I^{\eps,\Omega_0}(\rho,m)$. Then, up to a rearrangement, we may assume that $I^{\eps,\Omega_0}(\rho_k,m_k)\le -\alpha$. Using the perturbation property, Proposition~\ref{prop:PertuabationProperty}, for each pair $(\rho_k,m_k)$ and a standard diagonal argument, we find a sequence $\{\tilde\rho_k,\tilde m_k\}_k\subset X_0$ such that 
		\begin{align}
		(\tilde \rho_k,\tilde m_k)&\overset{d_X}{\to}(\rho,m) \quad \mathrm{in} \quad X,\\
		\lim_{k\to \infty}I^{\eps,\Omega_0} (\tilde \rho_k,\tilde m_k)&\ge I^{\eps,\Omega_0} (\rho, m)+\beta, 
		\end{align}  
		  for some $\beta>0$. This contradicts that $(\rho,m)$ is a point of continuity of $I^{\eps,\Omega_0}$ and thus if $(\rho,m)\in X$ is a point of continuity of $I^{\eps,\Omega_0}$, then $I^{\eps,\Omega_0}(\rho,m)=0$.
		  
		  Next, let $\Omega_k$ be an exhaustive sequence of bounded open subsets of $\Omega$. Consider the set $\Theta$ which is the intersection of 
		  \begin{equation}
		  \Theta_k:= \{ (\rho,m)\in X \; \colon \; I^{\frac{1}{k},\Omega_k} \;\; \text{is continuous at} \;\;(\rho,m)\}.
		  \end{equation}
		  Then $\Theta$ is the intersection of countably many residual sets and thus it is residual. Furthermore, if $(\rho,m)\in \Theta$, then $I^{\eps,\Omega_0}=0$ for any $\eps>0$ and any bounded $\Omega_0 \subset \Omega$. Applying Lemma \ref{lemma:I=0}, any $(\rho,m)\in \Theta$ satisfies the requirements of the theorem. 
	\end{proof}
	
	\begin{remark} 
	Our Theorem~\ref{prop:subsolutionconditiondens} can be seen as a general framework for constructing ``wild'' solutions to the compressible Euler equations as known from several recent works. To exemplify this, let us recall two important results in this direction:
	\begin{enumerate}
	\item In~\cite{chiodaroli}, the author exhibits energy-admissible solutions on the torus for any given initial density in $C^1$. In fact, her solutions are \emph{semi-stationary}, that is, their densities are time-independent. Noting that our Theorems~\ref{prop:subsolutioncondition} and~\ref{prop:subsolutionconditiondens} are equally valid on the torus, we recognise Chiodaroli's Proposition~4.1 as a special case of our Theorem~\ref{prop:subsolutionconditiondens}. Therefore, using Chiodaroli's construction of appropriate subsolutions and her argument for energy admissibility, we can recover her main result in our framework.
	\item A similar observation is true for~\cite{CDLK}, where it is proved that admissible wild solutions can emerge from certain  Riemann data. The main convex integration statement (Proposition 3.6 in~\cite{CDLK}) is easily seen to be a special case of our Theorem~\ref{prop:subsolutionconditiondens}, once we note that our result remains valid even if the subsolution is only pieceweise continuous. The energy admissibility of the solutions in~\cite{CDLK}, however, is facilitated by the fact that the specific subsolutions under consideration have piecewise constant density. 
	\end{enumerate}
	We hasten to add that the results in~\cite{chiodaroli} and~\cite{CDLK} do not immediately follow from our framework, but require in addition substantial efforts in the construction of appropriate subsolutions and the verification of the energy inequality, respectively.
	\end{remark}
		
	\section*{Appendix: A Toy Model}
	\setcounter{theorem}{0}
    \renewcommand{\thetheorem}{A\arabic{theorem}}
	
	Consider the following toy problem, which is a variant of the problem in~\cite[Section 6.1]{hprinc}: Find a pair of, say, $L^2$ functions $(u,v): [0,1]\to\R^2$ satisfying $|u|+|v|=1$ almost everywhere. There is no distributional differential constraint in this problem. 
	
	This is supposed to illustrate, in the simplest possible setting, the counter-intuitive effect that a somewhat symmetric result (in density and momentum) can arise from a completely asymmetric construction (because oscillations are added only to the momentum, not to the density) in the construction for the compressible Euler system. In our toy problem, one might think of $u$ as the velocity and $v$ as the pressure. The quantity $|u|+|v|$ can be thought of as the `total energy' or `generalized pressure', which should be 1 almost everywhere.  
	
	Set $X_0$ to be the space of (say, smooth) maps $(u,v):[0,1]\to\R^2$ that satisfy $|u|+|v|<1$ (`subsolutions'). Clearly, $X_0$ is norm-bounded in $L^2(0,1;\R^2)$, hence the weak topology in $L^2$ is metrizable by a metric $d$. Let $X$ be the $d$-closure of $X_0$, then $(X,d)$ is a complete metric space, and $|u|+|v|\leq 1$ for all $(u,v)\in X$.  
	
	Let the functional $I:X\to\R$ be given by
	\begin{equation}
	I(u,v):=\int_0^1[(|u|+|v|)^2-1]dx.
	\end{equation}
	Then (by convexity) $I$ is lower semicontinuous in $(X,d)$, and $I(u,v)\leq0$ with equality if and only if $(u,v)$ is a solution of our problem. By the Baire Theorem, the set of points of continuity of $I$ is a residual set in $X$. The following perturbation property shows that at every such point of continuity, $I=0$:
	
	\begin{prop}
	Let $\alpha>0$. Then there exists $\beta>0$ such that the following holds: For each $(u,v)\in X_0$ such that $I(u,v)<-\alpha$, there is a sequence $(u_k,v_k)\subset X_0$ converging (w.r.t.\ $d$) to $(u,v)$ which satisfies
	\begin{equation}
	\liminf_{k\to\infty}\;I(u_k,v_k)\geq I(u,v)+\beta.
	\end{equation}
	\end{prop}
	
	\begin{proof}
	Set $v_k=v$, $u_k(x):=u(x)+\frac12 \sin(kx)(1-(|u(x)|+|v(x)|)^2)$, then clearly $(u_k,v_k)\in X_0$, and
	\begin{equation}
	\begin{aligned}
	I(u_k,v_k)&=\int_0^1[(|u_k|+|v|)^2-1]dx\\
	&\geq  I(u,v)+\frac 14\int_0^1\sin^2(kx)[1-(|u|+|v|)^2]^2dx + \int_0^1 u(1-(|u|+|v|)^2)\sin(kx)dx.
	\end{aligned}
	\end{equation}
	The last term converges to zero and the second to
	\begin{equation}
	\frac 18\int_0^1[1-(|u|+|v|)^2]^2dx.
	\end{equation}
But this in turn can be estimated below, by Jensen's inequality, by
	\begin{equation}
	\frac 18 I(u,v)^2\geq \frac18 \alpha^2,
	\end{equation}
	so $\beta:=\frac{1}{16}\alpha^2$ does the job for sufficiently large $k$.
	\end{proof}
	
	This proves 
	
	\begin{theorem}\label{resi}
	The set of $(u,v)$ with $|u|+|v|=1$ a.e.\ is residual in $(X,d)$.
	\end{theorem}
	
	\begin{remark}
	This looks extremely strange at first glance for the following reason: {\em We have used no perturbation in the $v$ component in the proof whatsoever!} Yet, the theorem seems to imply that there are infinitely many solutions such that $u$ and $v$ are both (individually!) nowhere continuous. It seems counter-intuitive that $v$, which is entirely untouched by the given perturbations, should also have this property.
	
	Does this argument hold up to scrutiny? In order to show that a Baire-residual subset of $X$ contains infinitely many elements with nowhere continuous $v$, one needs to show that $X$ contains an open set of $v$'s (if, for instance, we had defined $X_0$ to be the space of $(u,v)$ with $|u|+|v|=1$ a.e.\ and additionally $v\equiv 0$, of course this would not have been the case). For the problem treated here, it is of course true, and so it is for the compressible Euler problem, thanks to the possibility of \emph{some} perturbation to $\rho$, even if the density perturbation is not required. 
	\end{remark}

\end{document}